\documentclass{amsart}
\usepackage{graphicx}  	
\usepackage{amsmath,amssymb} 
\usepackage{amsthm}
\usepackage{bm}  		
\usepackage{pdfsync}  	
\usepackage{subfigure}
\usepackage{color}

\usepackage[T1]{fontenc}
\usepackage{yfonts}

\theoremstyle{plain}
\newtheorem{theorem}{Theorem}[section]
\newtheorem{proposition}[theorem]{Proposition}
\newtheorem{lemma}[theorem]{Lemma}
\newtheorem{corollary}[theorem]{Corollary}

\theoremstyle{definition}
\newtheorem{example}[theorem]{Example}
\newtheorem{definition}[theorem]{Definition}
\newtheorem{remark}[theorem]{Remark}

\newcommand{\C}{{\mathbb C}}
\newcommand{\R}{{\mathbb R}}

\newcommand{\Sc}{{\mathcal S}}
\newcommand{\U}{{\mathcal U}}

\newcommand{\x}{{\mathbf x}}
\newcommand{\yy}{{\mathbf y}}

\newcommand{\aaa}{{\mathbf a}}
\newcommand{\eee}{{\mathbf e}}
\newcommand{\ddd}{{\mathbf d}}
\newcommand{\bfalpha}{{\bm \alpha}}
\newcommand{\bfeta}{{\bm \eta}}
\newcommand{\f}{{\mathbf f}}

\newcommand{\pp}{{\mathfrak p}}

\newcommand{\PP}{{\mathfrak P}}

\def\LT{\mathop{\rm LT}\nolimits}

\def\Spec{\mathop{\rm Spec}\nolimits}

\def\y1{\mathbf{y}^{(1)}}

\let\epsilon=\varepsilon
\let\rho=\varrho
\let\phi=\varphi
\let\To=\longrightarrow
\def\TTo#1{\mathop{\longrightarrow}\limits ^{#1}}

\def\longiso{\,\smash{\TTo{\lower 7pt\hbox{$\scriptstyle\sim$}}}\,}
\let\implies=\Rightarrow

\def\cocoa{\mbox{\rm
C\kern-.13em o\kern-.07 em C\kern-.13em o\kern-.15em A}}
\def\apcocoa{\mbox{\rm
A\kern-0.13em p\kern -0.07em C\kern-.13em o\kern-.07 em C\kern-.13em
o\kern-.15em A}}

\begin{document}

\title{An Algebraic Approach to Hough Transforms}

\author{Mauro C. Beltrametti \ and\  \  Lorenzo Robbiano}
\address{\scriptsize Dipartimento di Matematica, 
\ Universit\`a di Genova, \ Via
Dodecaneso 35,\
I-16146\ Genova, Italy}
\email{beltrame@dima.unige.it, robbiano@dima.unige.it}


\date{\today}

\begin{abstract}

The main purpose of this paper  is to lay the foundations 
of a general theory which encompasses the features of the classical 
Hough transform and extend them to general algebraic objects 
such as affine schemes. The main motivation comes from
problems of detection of special shapes in medical 
and astronomical images. The classical Hough transform
has been used mainly to detect simple curves such as lines and circles.
We generalize this notion using reduced Gr\"obner bases of
flat families of affine schemes.
To this end we introduce and develop the 
theory of {\it Hough regularity}. The theory is highly effective 
and we give some examples computed with \cocoa (see~\cite{Co}). 
\end{abstract}


\maketitle
\section{Introduction}
The Hough transform (or transformation) is a  technique mainly
used in image 
analysis  and digital image processing.
It was introduced by P.V.C. Hough in 1962  in the form 
of a patent (see~\cite{H}). 
Its original application was in physics  for 
detection of lines and arcs in the photographs 
obtained in particle detectors, and
many extensions and refinements of this method have 
been investigated since. In principle 
it can detect arbitrary shapes in images, given a 
parametrized description of the shape in question.

The main tool to achieve this result is a voting procedure 
in the parameter space. 
Roughly speaking, for instance  to detect line segments it 
works in the following way.
We can analytically describe a straight line  in a number of forms. 
Although in practice  the most convenient  representation is
the polar equation ${x\cos\theta +y\sin\theta =\rho}$, 
suppose that instead we use the cartesian equation $y=mx+n$.
For each point $(a_1,a_2)$ in the source image, 
a {\it potential}\/ line passing through it has 
parameters $m, n$ which necessarily satisfy
the equation $a_2 = ma_1+n$. 
This is the equation of  a straight line in the 
parameter space with coordinates $m,n$.
If many points  in the source image lie on the straight line $y=m_0x+n_0$, 
many lines in the parameter space will pass through the 
same {\it point} $(m_0, n_0)$.
A discretization of the parameter space into small 
cells and an {\it accumulator 
matrix}\/ whose entries count the number 
of lines passing through the cell will tell us 
that there is a maximum value corresponding to the cell 
containing $(m_0,n_0)$.
More details can be found for instance in the paper~\cite{DH}.

A key winning factor for this strategy is that in the plane, a 
straight line which is not parallel to the $y$-axis  has a well-defined 
unique representation of the  type  ${y=mx+n}$. 
Therefore, if a local maximum is obtained in the cell 
corresponding to $(m_0,n_0)$, the line of equation $y = m_0x+n_0$
is detected in the source image.
Moreover if two local maxima are obtained in different cells, 
they correspond to two {\it different lines}. This is a 
fundamental property which we call Hough regularity and which must be kept in every generalization.

In recent years, in particular in problems of recognition 
of special shapes in medical and astronomical images, 
much effort has been made to apply the above described procedure
to the detection of more complicated objects, in particular special 
algebraic plane and space curves.

In this paper we want to lay the foundations of a general 
theory which encompasses the features of the classical 
Hough transform and extends them to general algebraic objects 
such as affine schemes.

To this end, a family of algebraic schemes is required to have the 
property that its fibers are irreducible and
share the most essential properties such as the degree. 
So, we restrict ourselves to free families over an affine 
space and to achieve this property we use the notion of 
reduced Gr\"obner basis (see Section 2). 

Moreover, the uniqueness of the reduced Gr\"obner 
basis of an ideal, given a term 
ordering $\sigma$ (see~\cite{KR1}, Section 2.4.C),
implies the possibility of associating a 
well-defined set of coefficients to every fiber of the family. 

We define the Hough transform (Definition~\ref{Hough}) of a point, and
prove Theorem~\ref{cequalc} which describes the interplay 
between algebraic objects in the {\it source space}\/ and in the {\it parameter space}.
The main consequence is that we are in the position  of defining 
the key notion of Hough $\sigma$-regularity (see 
Definition~\ref{sigmaHoughregular}). It~describes the 
situation where equality of fibers implies equality of the 
corresponding parameters. 
In algebro-geometric terminology it means the following. If we consider the
uni-rational variety $V$ defined parametrically by the set
of coefficients of the reduced Gr\"obner bases of the generic fiber, 
the given parametrization represents $V$ as a rational variety.

Our main result is a criterion for detecting Hough $\sigma$-regularity
which is embodied in Thereom~\ref{main} which rests 
on Theorem~\ref{injectivity}.
It is then extended to the generic Hough regularity (see Therem~\ref{genericH}), while the special case where the fibers of our family are hypersurfaces is described in Remark~\ref{Hypersurfaces}.

Another important feature of our presentation is 
that Hough $\sigma$-regularity and generic Hough 
regularity are computable. The last section is devoted 
to the illustration of several examples computed with
the help of \cocoa\ (see~\cite{Co}).

Finally, special thanks are due to M. Piana and A.M. Massone who introduced us to the topic (see~\cite{BMP}) and to A.M. Bigatti who helped us in the implementation of the basic
\cocoa-functions used in Section~\ref{transform}.


\section{Families of Schemes}
\label{Families of Schemes}

We start the section by recalling some definitions.
The notation is borrowed
from~\cite{KR1} and~\cite{KR2}, in particular
we let~$x_1, \dots, x_n$ be
indeterminates and let $\mathbb T^n$
be the monoid of the power products in the
symbols $x_1, \dots, x_n$.
Most of the time, for simplicity we use the notation
$\x = x_1, \dots, x_n$.
If~$K$ is a field, the multivariate
polynomial ring~$K[\x]=K[x_1,\dots,x_n]$ is
denoted by~$P$, and if $f_1(\x), \dots, f_k(\x)$
are polynomials in $P$,
the set~$\{f_1(\x), \dots, f_k(\x)\}$ is denoted by~$\f(\x)$
(or simply by~$\f$).
Finally, we denote the {\it polynomial system}\/
associated to~$\f(\x)$  by~$\f(\x)=0$, 
and we say that the system is $d$-dimensional if 
the ideal generated by~$\f(\x)$ is $d$-dimensional
(see~\cite{KR1}, Section~3.7).

Let $n$ be a positive integer,
let~$P=K[x_1, \dots,x_n]$, 
let $\f(\x)=\{f_1(\x),\ldots,f_n(\x)\}$
define a $d$-dimensional irreducible, reduced scheme, 
and let~$I$ be the ideal of~$P$ generated by~$\f(\x)$.
We let~$ m$ be a positive integer
and let $\aaa = (a_1, \dots, a_m)$
be an $m$-tuple of indeterminates which will play the role of
parameters. If $F_1(\aaa, \x), \ldots, F_n(\aaa, \x)$ are polynomials
in~$K[\aaa, \x]$ we let~${F(\aaa, \x)=0}$ be the corresponding 
family of systems of  equations 
parametrized by~$\aaa$, and 
the ideal generated by~$F(\aaa,\x)$ in $K[\aaa,\x]$
is denoted by $I(\aaa,\x)$. 
If the scheme of the $\aaa$-parameters is denoted by $\Sc$,
then there exists a $K$-algebra homomorphism
$\phi: K[\aaa] \To K[\aaa,\x]/I(\aaa,\x)$ or, equivalently, 
a morphism of schemes $\Phi: \mathcal{F}  \To \Sc$ 
where $\mathcal{F}= \Spec(K[\aaa,\x]/I(\aaa,\x))$.

Although it is not strictly necessary for the theory, for our applications
it suffices to consider independent parameters. Here is 
the formal definition.

\begin{definition}\label{independparams}
If $\Sc= \mathbb A^m_K$ and 
${I(\aaa,\x)\cap K[\aaa] = (0)}$, then  the parameters~$\aaa$ are 
said to be  {\bf independent} with respect to $F(\aaa,\x)$, or simply 
independent if the context is clear.
\end{definition}

\bigskip
\goodbreak
\noindent
A  theorem called {\it generic flatness}\/ (see~\cite{E}, Theorem 14.4)
prescribes the existence of a non-empty Zariski-open
subscheme $\U$ of $\Sc$ over which the morphism of schemes
$\Phi^{-1}(\U) \To \U$ is {\it flat}. In particular, it is possible 
to explicitly compute a subscheme over which the morphism is free.
To do this, Gr\"obner bases reveal themselves as a fundamental tool.

\begin{definition}\label{iflat}
Let $F(\aaa, \x)$ be a family 
which contains a scheme defined by~$\f(\x)$.
Let $\Sc=\mathbb A^m_K$ be the scheme of the independent 
$\aaa$-parameters
and let 
$$
\Phi: \Spec(K[\aaa,\x]/I(\aaa,\x))  \To \Sc
$$ 
be the associated
morphism of schemes.
A dense Zariski-open subscheme~$\U$ of~$\Sc$
such that~${\Phi^{-1}(\U) \To \U}$ is
free (flat, faithfully flat), is said to be an~{\bf $I$-free 
(\hbox{$I$-flat}, $I$-faithfully\ flat}) subscheme of~$\Sc$ or
simply an~$I$-free ($I$-flat, $I$-faithfully flat) scheme.
\end{definition}

\begin{proposition}\label{flatness}
With the above assumptions and notation,
let $I(\aaa, \x)$ be the ideal 
generated by~$F(\aaa,\x)$ in $K[\aaa,\x]$, let $\sigma$ 
be a term ordering on~$\mathbb T^n$, let~$G(\aaa,\x)$ be the 
reduced~$\sigma$-Gr\"obner basis of the ideal~$I(\aaa, \x)K(\aaa)[\x]$, 
let~$d(\aaa)$ be the least common multiple of all the denominators  
of the coefficients of the polynomials in $G(\aaa,\x)$, 
and let~$T =\mathbb T^n\setminus \LT_\sigma(I(\aaa,\x)K(\aaa)[\x])$.

Then the open subscheme $\U$ of $\mathbb A^m_K$ 
defined by $d(\aaa)\ne 0$ is~$I$-free.
\end{proposition}

\begin{proof}
We consider the morphism
${\rm Spec}\big(K(\aaa)[\x]/I(\aaa,\x)K(\aaa)[\x]\big) 
\To  {\rm Spec}(K(\aaa))$. 
A standard result in Gr\"obner basis theory 
(see for instance~\cite{KR1}, Theorem 1.5.7) 
shows that the residue classes of the elements 
in~$T$ form a~$K(\aaa)$-basis 
of this vector space. We denote by $\U$ the open subscheme 
of~$\mathbb A^m_K$ defined by $d(\aaa) \ne 0$. 
For every point in $\U$, the given 
reduced Gr\"obner basis evaluates to the reduced 
Gr\"obner basis of the corresponding ideal. 
Therefore the leading term ideal is the 
same for all these fibers, and so is its complement~$T$. 
If we denote by~$K[\aaa]_{d(\aaa)}$ 
the localization of~$K[\aaa]$ at the 
element  $d(\aaa)$ and by $I(\aaa,\x)^e$ the 
extension of the ideal~$I(\aaa,\x)$ to the 
ring~$K[\aaa]_{d(\aaa)}$, 
then  $K[\aaa]_{d(\aaa)}[\x]/I(\aaa,\x)^e$ turns out 
to be a free~$K[\aaa]_{d(\aaa)}$-module.
This consideration concludes the proof.
\end{proof}

\begin{remark}\label{varie}
We observe that the term ordering  $\sigma$ can be 
chosen arbitrarily.
\end{remark}

\begin{example}\label{twopoints}
We let  $P = \C[x]$, the univariate polynomial ring,  and 
embed the ideal~$I$ generated by the polynomial
${x^2-3x+2}$ into the generically zero-dimensional
family~${F(\aaa, x) = \{a_1x^2 - a_2x + a_3\}}$.
Such  family is given by the canonical $K$-algebra
homomorphism 
$$
\phi:\!  \C[\aaa]\! \To\! \C[\aaa, x]/(a_1x^2 - a_2x + a_3)
$$ 
It is a zero dimensional complete intersection~for

$\{(\alpha_1, \alpha_2, \alpha_3) \in \mathbb A_\C^3 \ | \  
 \alpha_1\ne 0\}\  \cup \
\{(\alpha_1, \alpha_2, \alpha_3) \in  \mathbb A_\C^3 \ | \    
\alpha_1=0, \ \alpha_2 \ne 0\}$.

\noindent It represents two distinct smooth points for

$\{(\alpha_1, \alpha_2, \alpha_3) \in A_\C^3 \ | \   \alpha_1 \ne 0, \
\alpha_2^2-4\alpha_1\alpha_3 \ne 0\}$.

\noindent It represents
a smooth point for
$\{(\alpha_1, \alpha_2, \alpha_3) \in A_\C^3 \ | \   
\alpha_1=0,\ \alpha_2\ne 0\}$.

\end{example}

\section{The Hough Transform}
\label{transform}

We are not assuming that $K$ is algebraically closed, hence 
we must distinguish between maximal ideals and maximal linear 
ideals. The last ones  correspond to $K$-rational points.

Suppose that $\Phi: \mathcal{F} \To \mathbb A^m_K$ 
represents a dominant family of  
sub-schemes of $\mathbb A^n_K$ parametrized by $\mathbb A^m_K$
which corresponds to a $K$-algebra 
homomorphism ${\phi: K[\aaa] \To K[\aaa,\x]/I(\aaa,\x)}$.
The dominance implies that the~$\aaa$-parameters are independent, 
therefore  $\phi$ is injective.
If we fix $\bfalpha = (\alpha_1, \dots, \alpha_m)$, a 
rational ``parameter point'' in $\mathbb A^m_K$,
we get a special fiber of $\Phi$, 
namely $\Spec(K[\bfalpha,\x]/I(\bfalpha,\x))$, hence a special member 
of the family.  Clearly we have $K[\bfalpha,\x]=K[\x]$ so that 
$I(\bfalpha,\x)$ can be seen as an ideal in $K[\x]$. With this convention
we denote the scheme $\Spec(K[\x]/I(\bfalpha,\x))$ by~$C_{\bfalpha,\x}$.

On the other hand,  there exists another 
morphism $\Psi: \mathcal{F} \To \mathbb A^n_K$  which 
corresponds to the $K$-algebra 
homomorphism $\psi: K[\x] \To K[\aaa,\x]/I(\aaa,\x)$.
If we fix a rational point  $p=(\xi_1, \dots, \xi_n) \in \mathbb A_K^n$, we 
get a special fiber of the morphism~$\Psi$, 
namely $\Spec(K[\aaa,p]/I(\aaa,p))$.  Clearly we have $K[\aaa,p]=K[\aaa]$ 
so that ~$I(\aaa,p)$ can be seen as an ideal in $K[\aaa]$. 
With this convention we denote the scheme $\Spec(K[\aaa]/I(\aaa,p))$
by~$\Gamma_{\aaa,p}$. If $p \in C_{\bfalpha,\x }$ then 
the pair  $( C_{\bfalpha,\x }, p)$ will be called a {\bf pointed fiber} 
of $\Phi$.

\begin{remark}
We observe that a rational  zero of $I(\aaa,p)$ 
is an $m$-tuple~$\bfalpha$ such that ${f(\bfalpha, p) =0}$ for 
every $f(\aaa,p) \in I(\aaa,p)$. 
Therefore $I(\aaa,p)$  corresponds to 
the sub-scheme of $\mathbb A^m_K$ which parametrizes 
the fibers of $\Phi$ which contain the rational point~$p$.
\end{remark}

\begin{definition}\label{Hough}{\bf (The Hough Transform)}\\
We use the notation introduced above, we let
$\bfalpha = (\alpha_1, \dots, \alpha_m)\in \mathbb A^m_K$ 
and let ${p=(\xi_1, \dots, \xi_n)\in  A^n_K}$.  
Then the scheme $\Gamma_{\!\aaa, p}$ 
 is said to be 
the {\bf Hough transform} of the point $p$ with respect to the family $\Phi$.
If it is clear from the context, we simply say that the scheme $\Gamma_{\!\aaa, p}$ is
the {\bf Hough transform} of the point $p$.
\end{definition}

\begin{example}
Let $\aaa = (a,b)$, let $\x= (x,y)$,
and let $\mathcal{F} = \Spec(K[\aaa, \x]/ F_{\aaa,\x})$,
where 
$$F_{\aaa,\x}= y(x-ay)^2-b(x^4+y^4)$$
Then let $\bfalpha=(1,1)$. 
The corresponding fiber in $\mathbb A_K^2$ is 
$C_{(1,1),\,\x}$ which is defined by the polynomial
$F_{(1,1),\,\x}=  y(x-y)^2-(x^4+y^4)$.
The point $p=(0,1)$ belongs to $C_{(1,1),\,\x}$ and it 
corresponds to the curve $\Gamma_{\aaa,  (0,1)}$ which 
is defined by the polynomial
$F_{\aaa, (0,1)} = a^2-b$.
The curve $\Gamma_{\aaa, (0,1)}$  is the Hough transform of 
the point $(0,1)$.
\end{example}

\bigskip

Next, we are going to describe some properties of the 
Hough transforms.

\begin{proposition}\label{alfaalfa'}
Let $\Phi: \mathcal{F} \To \mathbb A^m_K$ 
be a dominant family of  
sub-schemes of $\mathbb A^n_K$ parametrized 
by $\mathbb A^m_K$
and let $\bfalpha$, $p$, $C_{\bfalpha,\x}$, $\Gamma_{\aaa, p}$ 
be as described above.

\begin{itemize}
\item[(a)] We have 
$\bfalpha \in \cap_{p\in C_{\bfalpha, \x} }{\Gamma_{\aaa,p} }$.
\item[(b)] 
If $\bfeta \in  \cap_{p\in C_{\bfalpha, \x} }{\Gamma_{\aaa,p} }$ 
then $C_{\bfalpha,\x} \subseteq C_{\bfeta,\x}$.
\end{itemize}
\end{proposition}
\proof
To prove (a) we observe that by definition 
$(\bfalpha, p)$ is a rational point of $\mathcal{F}$
for every $p \in C_{\bfalpha,\x}$ . This statement can be rephrased 
by saying that   $\bfalpha \in \Gamma_{\aaa,p}$ for 
every~$p \in C_{\bfalpha,\x}$.
Now we prove (b). To say 
that $\bfeta \in \cap_{p\in C_{\bfalpha, \x} }{\Gamma_{\aaa,p} }$
is equivalent to saying that $(\bfeta, p)$ is a rational 
point of $\mathcal{F}$ for all the points $p$ such 
that  $(\bfalpha, p)$ is a rational point of $\mathcal{F}$.
This fact can be rephrased by saying 
that $C_{\bfalpha,\x} \subseteq C_{\bfeta,\x}$.
\endproof

\begin{example}\label{real}
Let $\aaa = (a,b)$, let $\x= (x,y)$,
and let $\mathcal{F} = \Spec(\R[\aaa, \x]/ F_{\aaa,\x})$,
where 
$$F_{\aaa,\x}= x^2+ay^2-b$$
Then let $\bfalpha = (1,0)$. The corresponding fiber 
in $\mathbb A_K^2$ is 
$C_{(1,0),\,\x}$ which is defined by the polynomial
$F_{(1,0),\,\x}=  x^2+y^2$. There is only one rational 
point on it, i.e.\ $p = (0,0)$. It
corresponds to the curve $\Gamma_{(\aaa;\, 0,0)}$ 
which is defined by the polynomial~${F_{\aaa, (0,0)} = b}$.
Therefore $ \cap_{p\in C_{(1,0), \x} }{\Gamma_{\aaa,p} } 
= \Gamma_{(\aaa;\, 0,0)}$. 
We observe that  $\bfeta = (-1,0)$ belongs 
to $ \Gamma_{(\aaa;\, 0,0)}$
and that $C_{(-1,0),\x}$ is defined by the 
polynomial $F_{(-1,0),\x } = x^2-y^2$.
It consists of two lines which pass through the origin.
Consequently $C_{(1,0),\x} \subset C_{(-1,0),\x}$.
\end{example}

While Example~\ref{real} shows that in general 
$C_{\bfalpha,\x} \subset C_{\bfeta,\x}$  (see statement (b) of the above
proposition), we have a better result if $K$ is algebraically closed,
but we have to rule out some extremal cases, as the following 
example shows.

\begin{example}\label{twopoints}
Let $\aaa = a$, let $\x= x$,
and let $\mathcal{F} = \Spec(\C[a,x]/ F_{a,x})$,
where 
$$F_{a,x}= ax^2+x$$
Then let $\alpha = 0$. The corresponding fiber 
in $\mathbb A_\C^1(x)$ is 
$C_{0,x}$ which is defined by the polynomial
$F_{0,x}=  x$. There is only one rational 
point on it, i.e.\ $p = 0$. It
corresponds to the entire line $\mathbb A_\C^1(a)$ 
which is defined by the polynomial~${F_{a,0} = 0}$.
Therefore $ \cap_{p\in C_{0,x} }{\Gamma_{a,0} }= 
\mathbb A_\C^1(a)$. 
We observe that  $\eta = 1$ belongs 
to $\mathbb A_\C^1(a)$
and that~$C_{1,x}$ is defined by the 
polynomial $F_{1,x } = x^2+x$, so that
$C_{1,x}=\{0, 1\}$.
Consequently $C_{0,x} \subset C_{1,x}$.
\end{example}

As previously announced, we need to rule out extremal cases. 
Since the application that we have in mind deal with continuous families,
we resort to Proposition~\ref{flatness}.  
Moreover, instead of seeking the broadest generality, we 
mainly consider irreducible fibers of $\Phi$.
Finally, we use
degree-compatible term orderings, the reason being that 
the fibers in our families must share the affine Hilbert function. In particular, 
if the fibers are hypersurfaces we want that they share the degree.

\medskip

We start our investigation by recalling a straightforward result.

\begin{lemma}\label{equalradical}
Let $K[\x]$ be a polynomial ring over a field $K$ and let $I$, $J$ 
be ideals in~$P$ such that 
$\dim(P/I)=\dim(P/J)$, the radicals of  $I$ and $J$ are prime ideals,  
and $\sqrt{J}\subseteq \sqrt{I}$. 
Then $\sqrt{J} =  \sqrt{I}$.
\end{lemma}

\proof  
Let $d =\dim(P/I)$, let $\pp = \sqrt{J}$, and let $\PP = \sqrt{I}$. We have $\pp \subseteq \PP$.
On the other hand, we have $\dim(P/\pp) = \dim(P/\PP)$, hence $\pp = \PP$.
\endproof

\begin{remark}
The example with $I = (x^2,y)$, $J=(x,y^2)$ excludes the conclusion that $I=J$ in the previous lemma. The example $I=(x)$, $J=(x^2-xy)$ excludes the possibility of assuming $I$ and $J$ to be equidimensional even if they belong to the flat family defined by $x^2+axy=0$.
\end{remark}

\begin{definition}\label{restriction}
Let  $\Phi: \mathcal{F} \To \mathbb A^m_K$ be a dominant 
 family of sub-schemes of $\mathbb A^n_K$ parametrized 
 by $\mathbb A^m_K$. Then let $\sigma$  be
a degree-compatible 
term ordering on $\mathbb T^n$,
let $G(\aaa, \x)$ be the reduced $\sigma$-Gr\"obner basis
of the ideal $I(\aaa, \x)K(\aaa)[\x]$, and let $d_\sigma(\aaa)$ be the 
least common multiple of all the denominators of the coefficients 
of the polynomials in $G(\aaa,\x)$. We say that $d_\sigma(\aaa)$ is 
the {\bf $\sigma$-denominator} of $\Phi$. Moreover, 
let $\mathcal{U}_\sigma=A^m_K\setminus \{d_\sigma(\aaa) = 0\}$ and let 
$\Phi_{|d_\sigma(\aaa)}: \Phi^{-1}(\mathcal{U}) \To 
\mathcal{U}_\sigma$ be the corresponding restriction of $\Phi$.
We say that $\Phi_{|d_\sigma(\aaa)}$ is the {\bf $\sigma$-flat 
restriction} of $\Phi$.
\end{definition}

\begin{theorem}\label{cequalc}
Let $K$ be algebraically closed, let  $\Phi: \mathcal{F} \To \mathbb A^m_K$ be a dominant 
 family of sub-schemes of $\mathbb A^n_K$ parametrized 
 by $\mathbb A^m_K$, let $\sigma$  be
a degree-compatible 
term ordering on $\mathbb T^n$, let $d_\sigma(\aaa)$ be 
a $\sigma$-denominator of  $\, \Phi$, and let $\Phi_{|d_\sigma(\aaa)}$ be the
corresponding $\sigma$-flat restriction.
If we restrict to the family $\Phi_{|d_\sigma(\aaa)}$,  let $\bfalpha$, $p$, $C_{\bfalpha,\x}$, $\Gamma_{\aaa, p}$ 
be as described before,  
let ${\bfeta \in  \cap_{p\in C_{\bfalpha, \x} }{\Gamma_{\aaa,p} }}$, 
and assume that $C_{\bfalpha,\x}$ and $C_{\bfeta,\x}$ are irreducible,
then we have~${C_{\bfalpha,\x} =C_{\bfeta,\x}}$.
\end{theorem}

\goodbreak

\proof
We know that the morphism $\Phi_{|d_\sigma(\aaa)}$ is flat and that 
$G(\aaa,\x)$  specializes to the reduced $\sigma$-Gr\"obner 
basis of $I(\bfalpha,\x)$ for  every $\bfalpha \in \mathcal{U}_\sigma$ 
(see for instance the proof of Proposition 2.3 in~\cite{RT}).
By assumption, the term ordering $\sigma$ is 
degree-compatible, hence all the fibers of $\Phi_{|d_\sigma(\aaa)}$ 
share the same affine Hilbert function  (see Chapter 5 of~\cite{KR2}), hence
they share the same dimension.
On the other hand, we know that $C_{\bfalpha,\x} \subseteq C_{\bfeta,\x}$ by 
Proposition~\ref{alfaalfa'} and hence the assumption that $K$ is algebraically 
closed implies that $\sqrt{I(\eta, \x)} \subseteq \sqrt{I(\alpha, \x)}$. 
The assumption that $C_{\bfalpha,\x}$ and $C_{\bfeta,\x}$ are irreducible implies
that these two radical ideals are prime, hence we conclude
the proof using Lemma~\ref{equalradical}.
\endproof


\begin{definition}\label{sigmaHoughregular}
With the assumptions as in Theorem~\ref{cequalc}, 
let $\cap_{p\in C_{\bfalpha, \x} }{\Gamma_{\aaa,p} } =
\{\bfalpha\}$ for all $\bfalpha \in \mathcal{U}_\sigma$.
Then $\Phi_{|d_\sigma(\aaa)}$ is said to be {\bf Hough $\sigma$-regular}.

\end{definition}

An immediate consequence of the theorem is the following corollary.

\begin{corollary}\label{alfa}
With the same assumptions as in the theorem, the following conditions are equivalent
\begin{itemize}
\item[(a)] For all 
$\bfalpha, \bfeta \in \mathcal{U}$, the equality $C_{\bfalpha,\x} =C_{\bfeta,\x}$ 
implies $\bfalpha = \bfeta$.
\item[(b)] The morphism $\Phi_{|d_\sigma(\aaa)}$ is  Hough $\sigma$-regular.
\end{itemize}
\end{corollary}

\proof The theorem states that if $\bfeta \in  
\cap_{p\in C_{\bfalpha, \x} }{\Gamma_{\aaa,p} }$ 
then $C_{\bfalpha,\x} =C_{\bfeta,\x}$, 
hence the conclusion follows immediately.
\endproof

The meaning of this result is that under the assumption that 
$C_{\bfalpha,\x} =C_{\bfeta,\x}$ implies $\bfalpha = \bfeta$, the intersection 
of the Hough transforms of the pointed fibers $(C_{\bfalpha,\x}, p)$, 
when $p$ ranges through the points of $C_{\bfalpha,\x}$, turns out to 
be exactly $\{\bfalpha\}$, so that such intersection 
identifies the fiber $C_{\bfalpha,\x}$.
Therefore it is of great importance to detect situations where 
$\Phi_{|d_\sigma(\aaa)}$ is Hough $\sigma$-regular. 
For example, Hough regularity has been used as the conceptual 
basis for a pattern recognition algorithm (see~\cite{BMP}).
Hough $\sigma$-regularity does not hold in general as the 
following easy example shows.

\begin{example}\label{notregular}
Let $\aaa = a$, let $\x= x$,
and let $\mathcal{F} = \Spec(\C[a,x]/ F_{a,x})$,
where 
$$F_{a,x}= a^2x^2+x$$
Then $d(\aaa) = a$ (there is only one term ordering, 
so that we do not need to write~$d_\sigma(\aaa)$).
If we let $\alpha = 1$, the corresponding fiber 
in $\mathbb A_\C^1(x)$ is 
$C_{1,x}$ which is defined by the polynomial
$F_{0,x}=  x^2+x$. Therefore  $C_{1,x} = \{0,-1\}$
and $ \cap_{p\in C_{1,x} }{\Gamma_{a,p} }=
\Gamma_{a,0} \cap \Gamma_{a,-1}$
is defined by the polynomial~$F_a = a^2-1$.
We observe that  $\eta = -1$ is a zero of  $F_a$
and that~$C_{-1,x}=C_{1,x}$.
\end{example}

\goodbreak

\section{Detecting Hough-regularity}

After Corollary~\ref{alfa} the problem of finding Hough $\sigma$-regular families
rests on the ability of detecting families where ${C_{\bfalpha,\x} = C_{\bfeta,\x}}$
implies $\bfalpha=\bfeta$. The schematic meaning of ${C_{\bfalpha,\x} = C_{\bfeta,\x}}$
is $\sqrt{I(\bfalpha,\x)} = \sqrt{I(\bfeta,\x)}$, but we look for a more restricted condition, namely
we consider the following algebraic  problem.
When does the equality $I(\bfalpha,\x) = I(\bfeta,\x)$ imply $\bfalpha = \bfeta$?
To answer this question, we address a seemingly unrelated problem and prove some algebraic facts.
To do that, we make the following assumptions.

Let $K$ be an algebraically closed field and $m, s$ be two 
positive integers, let ${\aaa = (a_1, \dots, a_m)}$, $\yy = (y_1,\dots, y_s)$,
let $p_1(\aaa), \dots, p_s(\aaa), d_1(\aaa), \dots, d_s(\aaa)$ be polynomials in $K[\aaa]$, and let 
$d(\aaa) = {\rm lcm}(d_1(\aaa),\dots, d_s(\aaa) )$.
Then let $C$ be the affine  sub-scheme of~$\mathbb A^s_\yy$  given parametrically by 
$y_i = \frac{p_i(\aaa)}{d_i(\aaa)}$, let $D = \mathbb A^m_\yy \setminus \{d(\aaa) = 0\}$,
and denote the parametrization by $\mathcal{P}$.

\begin{definition}
With these assumptions, we say that $\mathcal{P}$ is injective if 
the corresponding morphism of schemes $D \To C$ is injective. In other words, 
$\mathcal{P}$ is injective if 
$\big( \frac{p_1(\bfalpha)}{d_1(\bfalpha)}, \dots, \frac{p_s(\bfalpha)}{d_s(\bfalpha)} \big)=
\big( \frac{p_1(\bfeta)}{d_1(\bfeta)}, \dots, \frac{p_s(\bfeta)}{d_s(\bfeta)} \big)$ 
implies $\bfalpha = \bfeta$. If this is the case,  $\mathcal{P}$ is a rational representation of $C$.
\end{definition}

\begin{definition}\label{idelta}
With the same assumptions, we double the set of 
indeterminates $\aaa = (a_1, \dots, a_m)$
by considering a new set of indeterminates $\eee = (e_1, \dots, e_m)$. 
Then we define two ideals in $K[\aaa, \eee]$, the ideal $I(Doub)$ generated by
$$\{p_1(\aaa)d_1(\eee) -  p_1(\eee)d_1(\aaa), \dots, 
p_s(\aaa)d_s(\eee) -  p_s(\eee)d_s(\aaa)\},$$
called the {\bf ideal of doubling coefficients of $\mathcal{P}$}, and
the ideal~$I(\Delta)$ generated by
$\{a_1-e_1, \dots, a_m-e_m\}$,
called the {\bf diagonal ideal}.
\end{definition}

The next theorem provides conditions under which a parametrization of the type described above is injective.

\begin{theorem}\label{injectivity}
Let $K$ be an algebraically closed field and $m,t$ be two 
positive integers, let ${\aaa = (a_1, \dots, a_m)}$, $\yy = (y_1,\dots, y_s)$,
let $p_1(\aaa), \dots, p_s(\aaa), d_1(\aaa), \dots, d_s(\aaa)$ be polynomials in $K[\aaa]$, and let 
$d(\aaa) = {\rm lcm}(d_1(\aaa),\dots, d_s(\aaa) )$.
Let $C$ be an affine rational sub-scheme of~$\mathbb A^s_\yy$ defined by the parametrization $\mathcal{P}$  given  by 
$y_i = \frac{p_i(\aaa)}{d_i(\aaa)}$, let $D = \mathbb A^m_\yy \setminus \{d = 0\}$,
and let $I(Doub)$ and $I(\Delta)$ be as described in Definition~\ref{idelta}.
Finally, let~$S(\Delta)$ be the saturation of~$I(Doub)$ with 
respect to~$I(\Delta)$.
Then the following conditions are equivalent.
\begin{itemize}
\item[(a)] The parametrization $\mathcal{P}$ is injective.
\item[(b)] The ideal $I(\Delta)$ is contained in the radical of the ideal $I(Doub)$.
\item[(c)] The ideal $I(\Delta)$ coincides with the radical of the ideal $I(Doub)$.
\item[(d)] We have $S(\Delta)=(1)$.
\end{itemize}
\end{theorem}

\proof
Clearly (b) and (d) are equivalent.
We observe that $I(Doub)\subseteq I(\Delta)$ and that $ I(\Delta)$ is a prime ideal. 
Therefore if (b) is satisfied then we 
have the inclusions $I(Doub)\subseteq I(\Delta) \subseteq \sqrt{I(Doub)}$. 
Passing to the radicals we get the chain of inclusions
$ \sqrt{I(Doub)}\subseteq I(\Delta) \subseteq \sqrt{I(Doub)}$ 
which proves that (b) implies (c) while the implication
(c) $\implies$ (b) is obvious. 
It is clear that $(\alpha_1, \dots, \alpha_m)$ and $(\eta_1, \dots, \eta_m)$
 yield the same image point in $C$
 if and only if $(\alpha_1, \dots, \alpha_m, \eta_1, \dots, \eta_m)$ is a zero of the
 ideal~$I(Doub)$. Consequently, condition (a) is equivalent to the 
 zero set  of $I(Doub)$ being contained in the zero set of $I(\Delta)$. 
 Since $K$ is algebraically closed, the Nullstellensatz implies that 
 this condition is equivalent to (b),  and the proof is complete.
 \endproof

\bigskip

Now we are ready to use these facts to construct a method for detecting Hough-regularity.

\begin{definition}
Let $\sigma$ be a term ordering on $\mathbb T^n$ and 
let $H$ be a tuple of polynomials in $K(\aaa)[\x]$.
If they are listed with $\sigma$-increasing
leading terms, we get a well-defined list of non constant
coefficients which is denoted by  $NCC_H$
and called the {\bf non constant coefficient list} of $H$.
\end{definition}

For example, il $\sigma ={\tt DegLex}$ and 
$H= (x_1x_2- \frac{{a_2^3-1}_{\mathstrut}}{a_1-a_2^2}x_2, \ \ 
x_1^3- \frac{{a_1^2}_{\mathstrut}}{a_1-a_2}x_1x_2 +
 \frac{{a_2^3}_{\mathstrut}}{a_1})$, we have 
$NCC_H= [ - \frac{{a_2^3-1}_{\mathstrut}}{a_1-a_2^2} ,\ 
 - \frac{{a_1^2}_{\mathstrut}}{a_1-a_2}, \ 
\frac{{a_2^3}_{\mathstrut}}{a_1} ]$.
\medskip

In the following definition we use the terminology of Definition~\ref{restriction}.
\begin{definition}
Let $\sigma$
be a degree compatible term ordering on $\mathbb T^n$ 
and let $G$ be 
the reduced $\sigma$-Gr\"obner basis of $I(\aaa, \x)K(\aaa)[\x]$,  
listed with $\sigma$-increasing leading terms. 
Let $NCC_G=(\frac{p_1(\aaa)}{d_1(\aaa)}, \dots, \frac{p_s(\aaa)}{d_s(\aaa)})$
be the non constant coefficient list of~$G$ and 
let ${d_\sigma(\aaa) = {\rm lcm}(d_1(\aaa),\dots, d_s(\aaa))}$ 
be the $\sigma$-denom\-inator of $\Phi$. 
Let $e_1, \dots, e_m$ be~${m}$ new  indeterminates, 
let $\eee= (e_1, \dots, e_m)$, and 
consider the following two ideals in the localization
$K[\aaa,\eee]_{d_\sigma(\aaa)\cdot d_\sigma(\eee)}$. 
The first ideal is generated by the $s$ polynomials
$$\{p_1(\aaa)d_1(\eee) -  p_1(\eee)d_1(\aaa), \dots, 
p_s(\aaa)d_s(\eee) -  p_s(\eee)d_s(\aaa)\},$$
is denoted by $I(DC_G)$, and called the 
{\bf ideal of doubling coefficients} of $G$.
The second ideal is generated by the $m$ polynomials
$\{a_1-e_1, \dots, a_m-e_m\}$,
is denoted by~$I(\Delta)$ and called the {\bf diagonal ideal}.

\end{definition}

\begin{theorem}\label{main}{\bf (Hough $\sigma$-regularity)}\\
Let $K$ be algebraically closed, let $\Phi: \mathcal{F} \To \mathbb A^m_K$ 
be a dominant family of  
sub-schemes of~$\mathbb A^n_K$ parametrized by~$\mathbb A^m_K$,
let $\sigma$
be a degree compatible term ordering on $\mathbb T^n$ and let~$G$ be 
the reduced $\sigma$-Gr\"obner basis of $I(\aaa, \x)K(\aaa)[\x]$ 
listed with $\sigma$-increasing leading terms. Then
let $d_\sigma(\aaa)$ be the $\sigma$-denominator of $\Phi$, let $I(DC_G)$ 
be the ideal of doubling coefficients of $G$, let~$I(\Delta)$
be the diagonal ideal, and let~$S(\Delta)$ be the saturation of~$I(DC_G)$ with 
respect to~$I(\Delta)$.
Then the following conditions are equivalent.
\begin{itemize}
\item[(a)] The morphism $\Phi_{|d_\sigma(\aaa)}$ is Hough $\sigma$-regular.
\item[(b)] The ideal $I(\Delta)$ is contained in the radical of the ideal $I(DC_G)$.
\item[(c)] The ideal $I(\Delta)$ coincides with the radical of the ideal $I(DC_G)$.
\item[(d)] We have $S(\Delta)=(1)$.
\end{itemize}
\end{theorem}

\proof
The equivalence of (b), (c), (d) follows as a special 
case of Theorem~\ref{injectivity}.
After Corollary~\ref{alfa} we observed that 
condition (a) can be expressed by saying that 
the equality $I(\bfalpha,\x) = I(\bfeta,\x)$ implies $\bfalpha = \bfeta$.
On the other hand, to say that $I(\bfalpha,\x) = I(\bfeta,\x)$ 
is equivalent to saying that
the reduced Gr\"obner bases of $I(\bfalpha,\x)$ 
and $I(\bfeta,\x)$ are identical. 
As we have already observed in the proof of 
Theorem~\ref{cequalc}, outside the hypersurface 
defined by $d_\sigma(\aaa) = 0$ the reduced \hbox{$\sigma$-Gr\"obner}
basis of~$I(\aaa, \x)K(\aaa)[\x]$ specializes to the reduced Gr\"obner bases
of the fibers. Therefore the equality of the defining ideals of two fibers
is equivalent to the equality of the corresponding coefficients in the 
 reduced Gr\"obner bases. Hence $(\alpha_1, \dots, \alpha_m)$
 yields the same reduced Gr\"obner basis as $(\eta_1, \dots, \eta_m)$
 if and only if $(\alpha_1, \dots, \alpha_m, \eta_1, \dots, \eta_m)$ is a zero of the
 ideal~$I(DC_G)$. Consequently, condition (a) is equivalent to the 
 zero set  of $I(DC_G)$ being contained in the zero set of $I(\Delta)$. 
Again the conclusion follows directly from Thereom~\ref{injectivity}.
\endproof

\begin{remark}\label{Hypersurfaces}{\bf (Hypersurfaces)}\\
Clearly, if the ideal $I(\aaa,\x)$ is principal all the matter is simplified.
Let $F=F(\aaa,\x)$ be a generator of $I(\aaa,\x)$.
In the case that all the coefficients of the leading 
form of~$F$ contain parameters, we have to choose $\sigma$ and
then invert the leading term of~$F$ to produce a monic polynomial 
$\overline{F}$. Then $\{\overline{F}\}$ is the reduced $\sigma$-Gr\"obner 
basis of $I(\aaa, \x)K(\aaa)[\x]$ and we can use the above theorem.
On the other hand, if one of the coefficients of  the leading 
form of~$F$ is constant, then the family is flat over the parameter 
space~$\mathbb A^m_K$. In this case we do not need 
to invert anything, hence if $\{p_1(\aaa),\dots, p_s(\aaa)\}$ is the 
list of non constant coefficients, we may consider the ideal generated
by $\{p_1(\aaa)-p_1(\eee), \dots, p_s(\aaa)-p_s(\eee)\}$
as the ideal of doubling coefficients, and then use the  theorem.
\end{remark}

Theorem~\ref{main}  yields a nice criterion to detect Hough $\sigma$-regularity. 
It depends on the term ordering chosen, hence it refers to a specific open
sub-scheme of $\mathbb A^m_K$. 
However, it turns out to be of particular importance the detection of cases where
the regularity is achieved generically, i.e.\ over a possibly different Zariski-open 
subschemes of the parameter space~$\mathbb A^m_K$.

\begin{definition}\label{genericHoughregular}
Let  $\Phi: \mathcal{F} \To \mathbb A^m_K$ be a dominant 
 family of sub-schemes of $\mathbb A^n_K$ parametrized 
 by $\mathbb A^m_K$ and let $\bfalpha$, $p$, 
 $C_{\bfalpha,\x}$, $\Gamma_{\aaa, p}$ 
be as described at the beginning of the section.
\begin{itemize}
\item[(a)] We say that the morphism $\Phi$ is {\bf generically Hough regular}
if there exists a non-empty open subscheme $\mathcal{U}$
of~$ \mathbb A^m_K$ such that 
$\cap_{p\in C_{\bfalpha, \x} }{\Gamma_{\aaa,p} } =
\{\bfalpha\}$ for all~$\bfalpha \in \mathcal{U}$. 
In this case we can say that $\Phi$ is Hough $\mathcal{U}$-regular.

\item[(b)] We say that the morphism $\Phi$ is {\bf Hough regular}
if it is generically Hough regular and $\mathcal{U}=\mathbb A^m_K$. 
\end{itemize}
\end{definition}

\begin{corollary}\label{houghextension}
With the same assumptions as in  Theorem~\ref{main}, we let 
$0\ne h(\aaa)$ be an element of  $K[\aaa]$,
let $\ddd = d_\sigma(\aaa)\cdot h(\aaa)\cdot d_\sigma(\eee)\cdot h(\eee)$,
and consider $I(\Delta)$, $I(DC_G)$, $S(\Delta)$ as ideals of $K[\aaa,\eee]_\ddd$.
Then the following conditions are equivalent.
\begin{itemize}
\item[(a)] The morhism $\Phi$ is Hough $\mathcal{U}$-regular 
where $\mathcal{U} =\mathbb A^m_K \setminus \{d_\sigma(\aaa)\cdot h(\aaa)=0\}$.
\item[(b)] The ideal $I(\Delta)$ is contained in the radical of the ideal $I(DC_G)$.
\item[(c)] The ideal $I(\Delta)$ coincides with the radical of the ideal $I(DC_G)$.
\item[(d)] We have $S(\Delta)=(1)$.
\end{itemize}
\end{corollary}

\proof
We observed  that outside the hypersurface 
of $\mathbb A^n_K$ defined by
the vanishing of the $\sigma$-denominator $d_\sigma(\aaa)$ of $\Phi$,
the reduced~$\sigma$-Gr\"obner basis of~$I(\aaa, \x)K(\aaa)[\x]$ 
specializes to the reduced $\sigma$-Gr\"obner bases of the fibers
(see the proof of  Theorem~\ref{cequalc}). Consequently, 
the equivalences in Theorem~\ref{main} 
extend, with the same proof, to the complement of every hypersurface of 
type $d_\sigma(\aaa)\cdot h(\aaa)$ and the 
corresponding localization $K[\aaa]_\ddd$.
\endproof

By comparing this definition with Definition~\ref{sigmaHoughregular}, 
we observe that the notion of Hough $\sigma$-regularity is a 
special instance of the notion of generic Hough regularity. 
Example~\ref{weakViviani} in the next section illustrates this relation.
Now the question is: how can we detect  generic Hough regularity?
Even more, if we discover generic Hough regularity, can we 
find $\mathcal U$ explicitly? The next theorem provides 
an answer to both questions.

\begin{theorem}\label{genericH}{\bf (Generic Hough regularity)}\\
With the same assumptions as in~Theorem~\ref{main}, suppose that 
the morphism $\Phi$ is not Hough $\sigma$-regular
and that the ideal  $S(\Delta)\cap K[\aaa]$ is different from zero. 
Then we have the following facts.

\begin{itemize}
\item[(a)] The morphism $\Phi$ 
is generically Hough regular.
\item[(b)] If $h(\aaa)$ is a non-zero element of  $S(\Delta)\cap K[\aaa]$
and $\mathcal{U} = 
\mathbb A^m_K\setminus\{d_\sigma(\aaa)\cdot h(\aaa)=0\}$ 
then $\Phi$ is Hough~$\mathcal{U}$-regular.
\end{itemize}

\end{theorem}

\proof 
Let $h(\aaa)$ be a non-zero element of $S(\Delta)\cap K[\aaa]$. 
We observe that $I(DC_G)$ and $I(\Delta)$ are invariant under the action of 
switching $a_1, \dots, a_m$  with  $e_1, \dots, e_m$. Therefore $S(\Delta)$ 
enjoys the same property which implies that $h(\eee)\in S(\Delta)$.
Consequently, if we let $\ddd = d_\sigma(\aaa)\cdot h(\aaa)
\cdot d_\sigma(\eee)\cdot h(\eee)$ we deduce 
that $S(\Delta)K[\aaa,\eee]_\ddd = (1)$ 
Therefore the two claims of the theorem follow 
directly from Corollary~\ref{houghextension}.
\endproof

{\bf Question}: Is it true that condition (a) implies that 
$S(\Delta)\cap K[\aaa]$ is different from zero?

\goodbreak

\bigskip\bigskip
\section{Examples and Code}
\label{transform}

The computation in the following examples uses \cocoa-5.
Here we see the basic functions which were written by Anna Bigatti.
\bigskip

\begin{verbatim}
Define NonConstCoefficients(X)
 If Type(X) = RINGELEM Then
    Return [C In Coefficients(X) | deg(den(C))<>0 Or deg(num(C))<>0];
 Elif Type(X) = LIST Then
    Return flatten([NonConstCoefficients(F) | F In X], 1);
 EndIf;
    Error("Unknown type");
EndDefine; -- NonConstCoefficients
\end{verbatim}

\medskip

\begin{verbatim}
Define IdealOfDoublingCoefficients(S, L, a_Name, e_Name, t_Name)
If L = [] Then Error("list is empty"); EndIf;
R := RingOf(num(L[1]));
phia := PolyAlgebraHom(R, S, IndetsCalled(a_Name, S));
phie := PolyAlgebraHom(R, S, IndetsCalled(e_Name, S));
DC := [phia(num(F))*phie(den(F))-phie(num(F))*phia(den(F)) | F In L];
LCM_DEN := LCM([den(F) | F In L]);
Inv := phia(LCM_DEN) * phie(LCM_DEN) * RingElem(S,t_Name) -1;
    Return Ideal(DC) + Ideal(Inv);
EndDefine; -- IdealOfDoublingCoeficients
\end{verbatim}

\goodbreak
\begin{example}\label{spaceline}{\bf (A space line)}\\
Let $\Phi:\mathcal{F}\To \mathbb A_\mathbb C^4$ be defined by the ideal
$I(\aaa,\x) = (x-a_1y-a_2z,\ x-a_3y-a_4z)$. If $\sigma = {\tt DegRevLex}$, the 
reduced  $\sigma$-Gr\"obner basis of $I(\aaa,\x)\mathbb C(\aaa)[\x]$ is 
$$G = \big(\, y +\frac{a_2-a_4}{a_1-a_3}z,\ \  
x +\frac{a_2a_3 -a_1a_4}{a_1-a_3}z\, \big) \hbox{ \ hence\  \ }
NCC_G= \big[\, \frac{a_2-a_4}{a_1-a_3},\  
\frac{a_2a_3 -a_1a_4}{a_1-a_3}\big]$$
Consequently $I(DC_G)$ is generated by the set
$ \big\{ 
(a_2-a_4)(e_1-e_3)\,-\,(e_2-e_4)(a_1-a_3),\\ (a_2a_3 -a_1a_4)(e_1-e_3)-
(e_2e_3 -e_1e_4)(a_1-a_3),\  (a_1-a_3)(e_1-e_3)\,t-1
 \big\}$. The ideal~$I(\Delta)$ is generated 
 by $\{a_1-e_1,\ a_2-e_2,\ a_3-e_3,\ a_4-e_4\}$.
 We ask \cocoa-5 to check if  $I(\Delta)$ is contained in the 
 radical of  $I(DC_G)$  and the answer is negative.
\medskip
Here we see the \cocoa-code.
\begin{verbatim}
N := 4;  R ::= QQ[a[1..N]];
S ::= QQ[a[1..N], e[1..N], t];
K := NewFractionField(R);
Use P ::= K[x,y,z];
I := Ideal(x-a[1]*y-a[2]*z,  x-a[3]*y-a[4]*z);
RGB := ReducedGBasis(I);
NCC := NonConstCoefficients(RGB);
Use S;
IDelta := ideal([a[i]-e[i] | i In 1..N]);
IDC := IdealOfDoublingCoefficients(S, NCC, "a", "e", "t");
IsInRadical(IDelta, IDC);
--false
\end{verbatim}
We conclude that the morphism $\Phi$ is not Hough $\sigma$-regular.
A bit of further easy investigation shows that indeed we
get the same line for instance if we assign the values  $(0,1,2,3)$ 
and $(0,1,1,2)$ to $(a_1,a_2, a_3,a_4)$.
\end{example}

\begin{example}\label{canspaceline}{\bf (A canonical space line)}\\
A completetly different situation happens when the space line is presented 
in canonical form. Let $\Phi:\mathcal{F}\To \mathbb A_\mathbb C^4$ be defined by
$I(\aaa,\x) = (x-a_1z-a_2,\ y-a_3z-a_4)$. If $\sigma = {\tt DegRevLex}$, the 
reduced  $\sigma$-Gr\"obner basis of $I(\aaa,\x)\mathbb C(\aaa)[\x]$ is clearly
$$G = \big(x-a_1z-a_2,\ y-a_3z-a_4\big)$$
The trivial conclusion is that the morphism $\Phi$ is Hough $\sigma$-regular.
Since there are no denominators, $\Phi$ is Hough regular.
\end{example}

\begin{example}\label{conic}{\bf (First conic)}\\
Let $\Phi:\mathcal{F}\To \mathbb A_\mathbb C^1$ be 
defined by the ideal
$I(a,\x) = (x^2-a^2y-a^3)$. If ${\sigma = {\tt DegRevLex}}$, the 
reduced  $\sigma$-Gr\"obner basis of $I(\aaa,\x)\mathbb C(a)[\x]$ is 
$G= (x^2-a^2y-a^3)$. There is no denominator, 
hence $\mathcal{U} = \mathbb A_\mathbb C^1$. Therefore
$NCC_G= \big[ a^2,  a^3\big]$,
and we do not need to invert anything, so that 
$I(DC_G)= (a^2-e^2, a^3-e^3)$.
We have~$I(\Delta) = (a-e)$, we ask \cocoa-5 to check 
if  $I(\Delta)$ is contained in the 
 radical of~$I(DC_G)$ and the answer is positive.
Here we see the \cocoa-code.

\bigskip

\begin{verbatim}
R ::=  QQ[a];
S ::= QQ[a,e,t];
K := NewFractionField(R);
Use P ::= K[x,y];
I := Ideal(x^2-a^2*y-a^3);
RGB := ReducedGBasis(I);
NCC := NonConstCoefficients(RGB);
Use S;
IDelta := ideal(a-e);
IDC := IdealOfDoublingCoefficients(S, NCC, "a", "e", "t");
IsInRadical(IDelta, IDC);
--true
\end{verbatim}
\end{example}
Therefore  the morphism $\Phi$ is Hough regular.

\begin{example}\label{secondconic}{\bf (Second conic)}\\
Let $\Phi:\mathcal{F}\To \mathbb A_\mathbb C^1$ be 
defined by the ideal
$I(a,\x) = (x^2-a^2y-a^4)$. If ${\sigma = {\tt DegRevLex}}$, the 
reduced  $\sigma$-Gr\"obner basis of $I(\aaa,\x)\mathbb C(a)[\x]$ is 
$G= (x^2-a^2y-a^4)$. There is no denominator, 
hence $\mathcal{U} = \mathbb A_\mathbb C^1$. Therefore
$NCC_G= \big[ a^2,  a^4\big]$,
and we do not need to invert anything, so that 
$I(DC_G)= (a^2-e^2, a^4-e^4)$.
We have~$I(\Delta) = (a-e)$, we ask \cocoa-5 to check 
if  $I(\Delta)$ is contained in the 
 radical of~$I(DC_G)$ and the answer is negative.
Here we see the \cocoa-code.
\begin{verbatim}
R ::=  QQ[a];
S ::= QQ[a,e, t];
K := NewFractionField(R);
Use P ::= K[x,y];
I := Ideal(x^2-a^2*y-a^4);
RGB := ReducedGBasis(I);
NCC := NonConstCoefficients(RGB);
Use S;
IDelta := ideal(a-e);
IDC := IdealOfDoublingCoefficients(S, NCC, "a", "e", "t");
IsInRadical(IDelta, IDC);
--false
\end{verbatim}
\end{example}
Therefore the morphism $\Phi$ is not Hough $\sigma$-regular.
Clearly we get the same conic for $a =1$ and $a=-1$.
\bigskip

The next example shows that the ideal $I(\aaa,\x)$ 
requires $d(\aaa)d(\eee) t-1$ among its generators.

\begin{example}\label{spheric}{\bf (A quartic curve)}\\
Let $\Phi:\mathcal{F}\To \mathbb A_\mathbb C^2$ be defined by 
$I(a,\x) = (x^2+y^2+z^2-1,\ a_1xy-a_2y^2-z)$. If ${\sigma = {\tt DegRevLex}}$, 
the reduced  $\sigma$-Gr\"obner basis of $I(\aaa,\x)\mathbb C(a)[\x]$ is 
$$G = ( xy -\frac{a_2}{a_1}y^2 -\frac{1}{a_1}z, \  \ 
x^2 +y^2 +z^2 -1, \  \ 
y^3+\frac{a_1^2}{h}yz^2 +
\frac{a_1}{h}xz +
\frac{a_2}{h}yz -
\frac{a_1^2}{h}y
)$$
where $h = a_1^2+a_2^2$.
Then 
$\mathcal{U} = \mathbb A_\mathbb C^2\setminus\{a_1h=0\}$.
We ask \cocoa-5 to check 
if  $I(\Delta)$ is contained in the 
 radical of  $I(DC_G))$ and the answer is positive.
Here we see the \cocoa-code.
\begin{verbatim}
N := 2;  R ::= QQ[a[1..N]];
S ::= QQ[a[1..N], e[1..N], t];
K := NewFractionField(R);
Use P ::= K[x,y,z];
I := Ideal(x^2+y^2+z^2-1, a[1]*x*y-a[2]*y^2-z);
RGB := ReducedGBasis(I);
NCC := NonConstCoefficients(RGB);
Use S;
IDelta := ideal([a[i]-e[i] | i In 1..N]);
IDC := IdealOfDoublingCoefficients(S, NCC, "a", "e", "t");
IsInRadical(IDelta, IDC);
--true
\end{verbatim}
Therefore the morphism $\Phi$ is Hough $\sigma$-regular.
If we had removed the last generator of $I(DC_G)$, i.e.\ the 
generator which imposes the invertibility of $d(\aaa)$ and $d(\eee)$,
the answer would have been {\tt false}.

\end{example}

\begin{example}\label{weakViviani}{\bf (A family of Viviani curves)}\\
Let $\Phi:\mathcal{F}\To \mathbb A_\C^2$ be defined by 
$I(\aaa,\x) = (a_2(z-a_1)^2+y^2-a_2a_1^2,\ x^2+y^2+z^2-4a_1^2)$.
Classical Vivian curves are obtained for $a_2=1$ (see~\cite{S}, p. 461).
 If ${\sigma = {\tt DegRevLex}}$, 
the reduced  $\sigma$-Gr\"obner basis of $I(\aaa,\x)\C(\aaa)[\x]$ is 
$$
G = (y^2+a_2z^2-2a_1a_2z,\  \ x^2+(1-a_2)z^2+2a_1a_2z-4a_1^2)
$$
There are no denominators in the coefficients, hence the family 
is flat all over $\mathbb A^2_K$. However the family is not Hough regular,
as we check with \cocoa.

Here we see the \cocoa-code.
\begin{verbatim}
N := 2;  R ::= QQ[a[1..N]];
S ::= QQ[t, e[1..N], a[1..N]];
K := NewFractionField(R);
Use P ::= K[x,y,z,t];
ID:=Ideal(a[2]*(z-a[1])^2 + y^2 - a[2]*a[1]^2,  
x^2 + y^2 + z^2 - 4*a[1]^2);
RGB := ReducedGBasis(ID);
NCC := NonConstCoefficients(RGB);
Use S;
IDelta := ideal([a[i]-e[i] | i In 1..N]);
IDC := IdealOfDoublingCoefficients(S, NCC, "a", "e", "t");
IsInRadical(IDelta, IDC);
--false
\end{verbatim}
At this point we compute the saturation of $I(DC)_G$ with respect to $I(\Delta)$
and its intersection with $\C[a_1, a_2]$.

\begin{verbatim}
Sat:=Saturation(IDC,IDelta);
-- ideal(a[2], e[2], e[1] +a[1], t -1)
Elim([e[1],e[2],t], Sat);
-- ideal(a[2])
\end{verbatim}

We get the ideal generated by $a_2$ which means 
that $\Phi$ is Hough $\mathcal{U}$-regular 
where~$\mathcal U = \mathbb A^2_\C\setminus \{a_2=0 \}$
(see Theorem~\ref{genericH}).
\end{example}

\bigskip
\goodbreak

\begin{example}\label{monomial}{\bf (A monomial curve)}\\
Let $\Phi:\mathcal{F}\To \mathbb A_K^2$ be defined parametrically by
$$x_1=a_1u^3,\ \  x_2=a_2u^4,\ \ x_3=u^5$$
By eliminating $u$ we get generators of the ideal $I(\aaa,\x)$.
If ${\sigma = {\tt DegRevLex}}$, 
the reduced  $\sigma$-Gr\"obner basis of $I(\aaa,\x)K(\aaa)[\x]$ is 
$$
G = (x_2^2-\frac{a_2^2}{a_1}\,x_1x_3,\quad  x_1^2x_2-a_1^2a_2x_3^2,\quad 
x_1^3-\frac{a_1^3}{a_2}\,x_2x_3)
$$
The family is not Hough-regular as the following \cocoa-code shows.

Here we see the \cocoa-code.
\begin{verbatim}
N := 2;  R ::= QQ[a[1..N]];
S ::= QQ[t, a[1..N], e[1..N]];
K := NewFractionField(R);
Use P ::= K[x[1..3],u];
L:=[3,4,5];
ID:=Ideal( x[1]-a[1]*u^(L[1]),   x[2]-a[2]*u^(L[2]),  x[3]-u^(L[3]));
E:=Elim([u], ID);
RGB := ReducedGBasis(E);
NCC := NonConstCoefficients(RGB);
Use S;
IDelta := ideal([a[i]-e[i] | i In 1..N]);
IDC := IdealOfDoublingCoefficients(S, NCC, "a", "e", "t");
IsInRadical(IDelta, IDC);
--false
\end{verbatim}

At this point we check the reduced Gr\"obner basis 
of $E$ and get  the following set
$
\{a_1^5-e_1^5,\ 
a_2^5-e_2^5,\ 
a_2^2e_1-a_1^2e_2,\  a_1^2a_2-e_1^2e_2,  \
a_2e_1^3-a_1^3e_2, \
a_1a_2^3-e_1e_2^3, \
{ta_1a_2e_1e_2-1}, \\
ta_1^2e_2^3-a_2,\ 
te_1^3e_2^2-a_1,\ 
te_1^2e_2^4-a_2^2,\ 
ta_1^4e_2^2-e_1^2,\ 
ta_1e_1e_2^6-a_2^4\}
$.
Just looking at the first two polynomials, we see that if we restrict our 
check to real numbers, we get $a_1=e_1$, $a_2=e_2$, hence we may conclude that
the family is {\it real Hough regular}. 
\end{example}

%
%
%
%

\bigskip
%


\end{document}